\documentclass[reqno,11pt]{amsart}

\usepackage{palatino}
\usepackage{amsmath}
\usepackage{amsfonts,amssymb,amsmath,latexsym,wasysym,mathrsfs,bbm,stmaryrd,esint}
\usepackage[all]{xy}
\usepackage[usenames]{color}
\usepackage{epsfig}
\usepackage[mathcal]{euscript}

\theoremstyle{plain}
\newtheorem{thm}{Theorem}[section]

\newtheorem{lem}[thm]{Lemma}
\newtheorem{prop}[thm]{Proposition}
\newtheorem{cor}[thm]{Corollary}

\theoremstyle{definition}
\newtheorem{defn}{Definition}[section]

\newtheorem{exmp}{Example}[section]

\theoremstyle{remark}
\newtheorem*{rem*}{Remark}
\newtheorem{rem}[thm]{Remark}

\theoremstyle{definition}

\numberwithin{equation}{section}

\long\def\symbolfootnote[#1]#2{\begingroup%
\def\thefootnote{\fnsymbol{footnote}}\footnote[#1]{#2}\endgroup}

\newcommand{\CC}{\mathbb{C}}
\newcommand{\RR}{\mathbb{R}}
\newcommand{\R}{\mathbb{R}}
\newcommand{\NN}{\mathbb{N}}
\newcommand{\EE}{\mathbb{E}}
\newcommand{\ra}{ \rightarrow}
\newcommand{\g}{\gamma}
\newcommand{\lsh}{{\rm LSH}}
\newcommand{\e}{\epsilon}

\begin{document}

\title{Hypercontractivity for Log-Subharmonic Functions}

\author{Piotr Graczyk$^{(1)}$}
\address{$(1)$ Universit\'e d'Angers, 2 boulevard Lavoisier, 49045 Angers cedex 01, France}
\email{Piotr.Graczyk@univ-angers.fr,Jean-Jacques.Loeb@univ-angers.fr}
\author{Todd Kemp$^{(2)}$}
\address{$(2)$ 2-175, MIT \; 77 Massachusetts Avenue, Cambrdige, MA \; 02139}
\email{tkemp@math.mit.edu}
\author{Jean-Jacques Loeb$^{(1)}$}
\author{Tomasz \.Zak$^{(3)}$}
\address{$(3)$ Institute of Mathematics, Wroclaw University of Technology, 50--370 W. Wyspianskiego, Wroclaw, Poland}
\email{Tomasz.Zak@pwr.wroc.pl}

\date{\today}

\begin{abstract}
We prove  strong hypercontractivity (SHC) inequalities for logarithmically subharmonic functions on $\RR^n$ and different classes of measures: Gaussian measures on $\RR^n$, symmetric Bernoulli and symmetric uniform probability measures on $\RR$, as well as their convolutions.  Surprisingly, a slightly weaker strong hypercontractivity property holds for {\em any} symmetric measure on $\RR$.  For all measures on $\R$ for which we know the (SHC) holds, we prove that a log--Sobolev inequality holds in the log-subharmonic category with a constant {\em smaller} than the one for Gaussian measure in the classical context.  This result is extended to all dimensions for compactly-supported measures.
\end{abstract}

\maketitle

\symbolfootnote[0]{(2) This work was partially supported by NSF Grant DMS-0701162.}

\vspace{-0.4in}

\section{Introduction} \label{sect intro}

In this paper, we prove some important inequalities -- strong hypercontractivity (SHC) and a logarithmic Sobolev inequality -- for logarithmically subharmonic  functions (cf.\ Definition \ref{logsh} below.)  Our paper is inspired by work of Janson \cite{j}, in which he began the study of an important property of semigroups called {\bf strong hypercontractivity}.  A rich series of subsequent papers by Janson \cite{j2}, Carlen \cite{c}, Zhao \cite{z}, and recently by Gross (\cite{g,g2} and a survey \cite{g3}) was devoted to this subject on the spaces  $\CC^n$ and, in papers by Gross, on complex manifolds. In contrast to all the aforementionned papers, our results concern   the real spaces $\RR^n$.

\medskip

In the first part of the paper (Sections \ref{sect Hyp for Gaussian}--\ref{sect convolution}) we prove strong hypercontractivity in the log--subharmonic setting: for $0<p\le q<\infty$,
\begin{equation} \tag{SHC} \|T_t f\|_{L^q(\mu)}\le \|f\|_{L^p(\mu)} \quad \text{for} \quad t\ge \frac{1}{2}\log \frac{q}{p} \end{equation}
for the dilation semigroup $T_tf(x)=f(e^{-t}x)$, for any logarithmically subharmonic function $f$, for different classes of measures $\mu$: including Gaussian measures and some compactly supported measures on $\RR$ (symmetric Bernoulli and uniform probability measure on $[-a,a]$ for $a>0$). We also show that, in numerous important cases, the convolution of two measures satisfying (SHC) also satisfies (SHC).

\medskip

Let us note that in the theory of hypercontractivity for general measures, the semigroup considered is the one associated to the measure by the usual technology of Dirichlet forms.  The generator of the semigroup (on a complete Riemannian manifold) takes the form $-\Delta + X$ where $\Delta$ is the Laplace-Beltrami operator and $X$ is a vector field; hence, the semigroup restricted to {\em harmonic functions} on the manifold is simply the (backward) flow of $X$.  For Gaussian measure, $X = x\cdot\nabla$, yielding the above flow $T_t$; this vector field is often called the {\em Euler operator}, denoted $E$.  In a sense, the point of this paper is to show that the strong hypercontractivity theorems about this flow extend beyond harmonic functions to the larger class of logarithmically subharmonic functions.

\medskip

In the second part of the paper (Section \ref{sect Log Sobolev}) we show that a log--Sobolev inequality (LSI) in the log--subharmonic domain holds for Gaussian measure on $\RR$ and for all $1$-dimensional measures which satisfy the strong hypercontractivity (SHC) considered in the first part.  We also prove the general implication (SHC) $\Rightarrow$ (LSI) for compactly supported measures on $\RR^n$, still for log--subharmonic functions.  In both cases, the  (LSI) we get is {\em stronger} than the classical one in the following sense. Let
\[ t_N(p,q) = \frac{1}{2}\log\frac{q-1}{p-1}, \quad  t_J(p,q)=\frac{1}{2}\log\frac{q}{p} \]
denote the Nelson and Janson times (cf. \cite{n,j}), for $1<p\le q<\infty$ (in fact, $t_J$ makes sense for all positive $p\le q$). The classical hypercontractivity for $t\ge c\, t_N$ is equivalent, by Gross's theorem in \cite{g4}, to a logarithmic-Sobolev inequality with the constant $2c$:
\[ \int |f|^2\log |f|^2 d\mu -\|f\|_{2,\mu}^2 \log\|f\|_{2,\mu}^2 \le 2c \int f Lf d\mu \]
where $L$ is the positive generator of the semigroup.  We show that, in the category of logarithmically subharmonic functions, the strong hypercontractivity for $t\ge c\, t_J$  implies (LSI) with {\em constant $c$}:
\begin{equation} \tag{LSI} \int |f|^2\log |f|^2 d\mu -\|f\|_{2,\mu}^2 \log\|f\|_{2,\mu}^2 \le c\int f Ef d\mu \end{equation}
where $E$ is the Euler operator discussed above.  Hence, one cannot obtain this stronger LSI by simply restricting the classical Gaussian LSI to log--subharmonic functions.

\medskip

Let us note that the implication (SHC) $\Rightarrow$ (LSI) in the log--subharmonic case does not follow as easily as in the classical setting.  Indeed, if $f$ is log--subharmonic, the functions $f|_{[-N,N]}$ and $f{\bf 1}_{|f|< N}$ are  not log--subharmonic on $\R$, and the classical techniques of approximation  by more regular (e.g.\ compactly  supported or bounded) functions fail. Instead, the present approach is to approximate probability measures (e.g.\ Gaussian measures) by measures with compact support.  This requires proving some stronger versions of the DeMoivre--Laplace  Central Limit Theorem. These results, contained in Section \ref{CLT}, are interesting independently.

\medskip

Our principal reference for the basic preliminaries is the book \cite{a} which gives a very accessible survey on hypercontractivity and on logarithmic--Sobolev inequalities.

\bigskip

{\bf Acknowledgment}. We thank  A. Hulanicki for calling the attention of the first and third authors to hypercontractivity problems in the holomorphic category.  Thanks also go to L. Gross for many helpful conversations.

\section{log--subharmonic functions}

\begin{defn}\label{logsh}  An $L^1_{\mathrm{loc}}$ upper semi-continuous function $f: {\RR}^n \to [-\infty, +\infty)$, not identically equal to $-\infty$,  is called {\bf subharmonic} if for every $x,y\in{\RR}^n$, one has the inequality: 
\begin{equation}\label{mean}
f(x)\leq \fint_{O(n)} f(x+\alpha y)\,d\alpha
\end{equation} 
where $O(n)$ is the orthogonal group of  ${\RR}^n$ and $d\alpha$ is the normalized Haar measure on it. (The notation $\fint$ is a reminder that the measure in question is normalized.)  A non-negative function $g: {\RR}^n \to [0, +\infty)$ is called {\bf log--subharmonic} (abbreviated \lsh) if the function $\log g$ is subharmonic.
\end{defn} 

\begin{rem} \label{remark spherical average} Definition \ref{logsh} is evidently equivalent to insisting that $f(x) \le \fint_{\partial B(x,r)} f(t)\,\sigma(dt)$ for every $x\in\R^n$, where $\partial B(x,r)$ is the sphere of radius $r$ about the point $x$, and $\sigma$ is normalized Lebesgue measure on this sphere.  Frequently, subharmonicity is stated in terms of averages over solid balls $B(x,r)$ instead; the two approaches are equivalent for $L^1_{\mathrm{loc}}$ upper-semicontinuous functions. Subharmonic function (and ergo log-subharmonic functions) need not have very good local properties.  There are subharmonic functions that are discontinuous everywhere (see, for example, \cite{Sudallev}).  In some of what follows, it will be convenient to work with {\em continuous} \lsh\ functions; where this restriction is in place, we have stated it explicitly.  \end{rem}

\begin{exmp} The following examples of \lsh\ functions are well-known and easily verified.
\begin{enumerate}
\item A convex function is subharmonic. On $\R$, $f$ is subharmonic if and only if $f$ is convex. 
\item Let  $f$ be a holomorphic function on ${\CC}^n$. Then $|f|$ is  a log--subharmonic function (see \cite{h} or use Jensen's inequality).  Indeed, $\log|f|$ is actually harmonic on the complement of $\{f=0\}$.
\item Denote by $\langle\;,\,\rangle$ the scalar product on ${\RR}^n$, and fix $a\in {\RR}^n$. Then $x\mapsto \exp\langle a,x\rangle $ is a log--subharmonic function.
\end{enumerate}
\end{exmp}

\noindent The main content of the next proposition is item \ref{sum lsh}, which takes some work to prove and will be important in what follows.

\begin{prop}\label{properties} Let $f,g$ be \lsh, and let $p>0$.
\begin{enumerate}
\item \label{prod lsh} The product $fg$ is \lsh, as is $g^p$.
\item \label{sum lsh} The sum $f+g$ is \lsh.
\item \label{lsh -> sh} $f$ is subharmonic.
\end{enumerate}
\end{prop}

\begin{proof}
Property \ref{prod lsh} is evident. In order to prove \ref{lsh -> sh} (note that non-negativity is built into the definition of \lsh\ functions), we use the fact that if a function $\varphi:\RR\ra\RR$ is increasing and convex and $h$ is a subharmonic function then $\varphi(h)$ is also subharmonic.  We apply this fact with $\varphi(x)=e^x$ and $h=\log f$ when $f$ is \lsh.  To prove \ref{sum lsh}, we need the following lemma.

\begin{lem}\label{lemma 2 var convex}
Let $\varphi: \RR^2\ra \RR$ be a convex function of two variables, increasing in each variable. If $F$ and $G$ are subharmonic functions then $\varphi(F,G)$ is also subharmonic. 
\end{lem}

\begin{proof}
As the function $\varphi$ is convex, the region $R=\{(x,y,z)\in\R^3\,;\,\varphi(x,y)\ge z\}$ is convex; accordingly $R$ may be specified by a collection of tangent planes.  That is, there is a set of affine functionals $A_k(x,y) = a_k x + b_k y + c_k$ (for some constants $a_k,b_k,c_k\in\R$) ranging over some (typically uncountable) index set $k\in K$, such that $(x,y,z)\in R$ if and only if
$z \ge \sup_{k\in K} A_k(x,y)$.  Hence, the function $\varphi$ is determined by
\[ \varphi(x,y)=\sup_{k\in K} A_k(x,y) = \sup_{k\in K}(a_kx+b_ky+c_k). \]
The function $\varphi$ is increasing in each variable, so $a_k,b_k\geq 0$. Now, for $v\in\R^n$ and $r>0$, denote 
\[ P_{x,r}f=\fint_{\partial B(x,r)} f(t)\,\sigma(dt), \]
following the alternative condition for subharmonicity in Remark \ref{remark spherical average}.  To prove the lemma, it therefore suffices to show that $P_{x,r}\varphi(F,G)\geq \varphi(F,G)(x)$ for each $x$. We have
\begin{eqnarray*}
P_{x,r}\varphi(F,G)&=&P_{x,r}\sup_{k\in K} A_k(F,G)\geq
\sup_{k\in K}P_{x,r}(a_kF+b_kG+c_k)\\
 &\geq&  \sup_{k\in K}(a_kF(x)+b_kG(x)+c_k)
=\varphi(F,G)(x).
\end{eqnarray*}  
This proves the Lemma.
\end{proof}

\noindent It is easy to verify that the function $\varphi(x,y)=\log(e^x + e^y)$ satisfies the hypotheses of the lemma: to check its convexity, we write $\log(e^x+e^y)=x+\log(1+e^{x-y})$, yielding the result since the function $t\mapsto \ln(1+e^t)$ is convex.  Hence, if $f$ and $g$ are \lsh, then $f = e^F$ and $g = e^G$ for subharmonic functions $F,G$, and so the lemma yields that $\varphi(F,G) = \log(f+g)$ is subharmonic.  This ends the proof of the proposition.
\end{proof}

The following two corollaries of Proposition \ref{properties} are useful in much of the following.  

\begin{cor} \label{cor average 1}
Let $\Omega$ be a separable metric space, and let $\mu$ a Borel probability measure on $\Omega$.  Suppose $f\colon \R^n\times \Omega\to \R$ satisfies
\begin{enumerate}
\item The function $x\mapsto f(\omega,x)$ is \lsh\ and continuous for $\mu$--almost every $\omega\in\Omega$.
\item The function $\omega\mapsto f(\omega,x)$ is bounded and continuous for each $x\in\R^n$.
\item For small $r>0$, there is a constant $C_r>0$ so that, for all $\omega\in\Omega$ and all $x\in\R^n$, $|f(\omega,t)|\le C_r$ for $t\in B(x,r)$.
\end{enumerate}
Then the function $\tilde{f}(x) = \int_{\Omega} f(\omega,x)\,\mu(d\omega)$ is \lsh.
\end{cor}

\begin{proof} By Varadarajan's theorem (see Theorem 11.4.1 in \cite{Dudley}), there is a sequence of points $\omega_j\in\Omega$ such that the probability measures
\[ \mu_n = \frac{1}{n}\sum_{j=1}^n \delta_{\omega_j} \]
converge weakly to $\mu$: $\mu_n \rightharpoonup \mu$.  Note that
\[ \tilde{f}_n(x) = \int_\Omega f(\omega,x) \mu_n(d\omega) = \frac{1}{n}\sum_{j=1}^n f(\omega_j,x), \]
and by Proposition \ref{properties} part (\ref{sum lsh}), $\tilde{f}_n$ is \lsh\ for each $n$.  Moreover, since $f(\,\cdot\,,x)\in C_b(\Omega)$, weak convergence guarantees that $\tilde{f}_n(x) \to \tilde{f}(x)$ for each $x$.  Fix $\e>0$; then since $\tilde{f_n}$ and $\tilde{f}$ are non-negative, $\tilde{f_n}+\e$ and $\tilde{f}+\e$ are strictly positive and thus $\log(\tilde{f_n}(x)+\e)\to \log(\tilde{f}(x)+\e)$ for each $x$.  Again using Proposition \ref{properties}, $\tilde{f_n}+\e$ is \lsh\, and so $\log(\tilde{f_n}+\e)$ is subharmonic.  Let $r>0$ be small, and consider
\[ \fint_{\partial B(x,r)} \log(\tilde{f}(t)+\e)\,dt = \fint_{\partial B(x,r)} \lim_{n\to\infty} \log(\tilde{f_n}(t)+\e)\,dt. \]
By assumption, $|f(\omega,t)| \le C_r$ for each $\omega\in\Omega$ and $t\in \partial B(x,r)$; hence, $|\tilde{f_n}(t)|\le C_r$ a well.  This means there is a uniform bound on $\log(\tilde{f_n}+\e)$ on $\partial B(x,r)$.  We may therefore apply the dominated convergence theorem to find that
\[ \begin{aligned} \fint_{\partial B(x,r)} \log(\tilde{f}(t)+\e)\,dt
&= \lim_{n\to\infty} \fint_{\partial B(x,r)} \log(\tilde{f_n}(t)+\e)\,dt  \\
&\ge \lim_{n\to\infty} \log(\tilde{f_n}(x)+\e) = \log(\tilde{f}(x)+\e),
\end{aligned} \]
where the inequality follows from the fact that $\log(\tilde{f_n}+\e)$ is subharmonic.  Hence, $\tilde{f}+\e$ is \lsh\ for each $\e>0$.  Finally, since $f(\omega,x)$ is continuous in $x$ for almost every $\omega$, the boundedness of $f$ in $\omega$ shows that $\tilde{f}$ is continuous.  Thus the set where $\tilde{f}>-\infty$ is open.  Therefore $\log(\tilde{f}(x)+\e)$ is uniformly-bounded in $\e$ on small enough balls around $x$, and a simple argument like the one above shows that the limit as $\e\downarrow 0$ can be performed to show that $\tilde{f}$ is \lsh\ as required.  \end{proof}

\begin{rem} It is possible to dispense with the requirement that $f(\omega,x)$ is continuous in $x$ by using Fatou's lemma instead of the dominated convergence theorem; however, the continuity of $f(\omega,x)$ in $\omega$ is still required for this argument.  In all the applications we have planned for Corollary \ref{cor average 1}, $f(\omega,x)$ is such that continuity in one variable implies continuity in the other, and so we need not work harder to eliminate this hypothesis. \end{rem}

\begin{rem} \label{remark average sh} In Corollary \ref{cor average 1}, if \lsh\ is replaced with the weaker condition {\em lower-bounded subharmonic} (in the premise and conclusion of the statement), then the result follows from Definition \ref{logsh} with a simple application of Fubini's theorem; moreover, the only assumption needed is that $f(\,\cdot\,,x)\in L^1(\Omega,\mu)$ for each $x$. 
\end{rem}

\begin{cor} \label{cor average 2}
Suppose $f\colon\R^n\to\R$ is lower-bounded and subharmonic.  Then the function
\[ \tilde{f}(x) =  \int_{O(n)} f(\alpha x)\,d\alpha \]
is subharmonic.  Moreover, if $f$ is also \lsh\ and continuous, then so is $\tilde{f}$.  In either case, $\tilde{f}$ depends only on the radial direction: there is a function $g\colon[0,\infty)\to[-\infty,\infty)$ with $\tilde{f}(x) = g(|x|)$, and $g$ is {\em non-decreasing} on $[0,\infty)$.
\end{cor}

\begin{proof} Suppose $f$ is \lsh\ and continuous. The reader may readily verify that the function $(\alpha,x) \mapsto f(\alpha x)$ satisfies all the conditions of Corollary \ref{cor average 1}.  (The weaker statement for lower-bounded subharmonic $f$, not necessarily continuous, follows similarly via Remark \ref{remark average sh}.)  Clearly averaging $f$ over rotations makes $\tilde{f}$ radially symmetric.  Any radially symmetric subharmonic function is radially non-decreasing, by the maximum principle.
\end{proof}



\section{Hypercontractivity inequalities for the Gaussian measure} \label{sect Hyp for Gaussian}

Let $m$ be   a probability measure  on $ \RR^n$.  For $p\geq 1$, we denote the norm on $L^p(m)$ by $\| \quad  \| _{p,m}$.  We will denote by $L^p_\lsh(m)$ the cone of  log-subharmonic functions in $L^p(m)$.  Let $\g$ be the standard Gaussian measure on ${\RR}^n$, i.e.\ $\g(dx)=c_n \exp (-|x|^2 / 2)\, dx$, where  $dx$ is Lebesgue measure and $c_n= (2\pi)^{-n/2}$. 

Given a function $f$ on ${\RR}^n$, and $r\in [0,1]$, we denote by $f_r$ the function $x\mapsto f(rx)$.  The family of operators $S_rf=f_r, r\in [0,1]$ is a multiplicative semigroup, whose additive form $T_t f(x) = f(e^{-t}x)$ is considered in connection with holomorphic function spaces in \cite{c,g,j,z} and others (including the second author's paper \cite{k} in the non-commutative holomorphic category).  When $f$ is differentiable, the infinitesimal generator $E$ of $(T_t)_{t\geq 0}$ equals $-Ef$
where $E$ is the Euler operator
\[ Ef(x)=x\cdot \nabla f. \]
If $N$ is the Ornstein--Uhlenbeck operator $N= -\Delta + E$ acting in $L^2(\CC^n,\gamma)$ and $f$ is a holomorphic function then
$Nf=Ef$, so $(T_t)_{t\geq 0}$ and, equivalently, $(S_r)_{r\in [0,1]}$ act on holomorphic functions as the Ornstein--Uhlenbeck semigroup $e^{-tN}$ (cf. \cite{a} p.22--23).

\medskip

Before showing the strong hypercontractivity  of the  semigroup $S_r$ for the Gaussian measure and  \lsh\ functions, let us show that the operators $S_r$ are $L^p$--contractions on non-negative subharmonic functions, for any rotationally invariant probability measure. 
\begin{prop}\label{contraction} 
Let  $m$  be a probability measure on  ${\RR}^n$ which is $O(n)$-invariant. Then for $f\ge 0$ subharmonic , $r\in[0,1]$, and $p\ge 1$, we have
\[ \|f_r \|_{p,m}\leq  \|f \|_{p,m}. \]
Moreover, this contraction property holds additionally in the regime $0<p<1$ of $f$ is \lsh.
\end{prop} 
\begin{proof}
First consider the case $p\ge 1$, and assume only that $f\ge 0$ is subharmonic.  Note that, since $f\ge 0$ and since $m$ is $O(n)$-invariant,
\[ \|f_r\|_{p,m}^p = \int_{\R^n} f(rx)^p\,dm(x) = \int_{O(n)} \int_{\R^n} f(rx)^p\,dm(\alpha x)\,d\alpha. \]
Changing variables using the linear transformation $\alpha$ in the inside integral and using Fubini's theorem, we have (replacing $\alpha^{-1}$ with $\alpha$ in the end)
\[ \int_{\R^n} \int_{O(n)} f(r\alpha x)^p\,d\alpha\,dm(x) = \int_{\R^n} S_r\,h(x)\,dm(x), \]
where $h(x) = \int_{O(n)} f(\alpha x)^p\,d\alpha$; i.e., with $k = f^p$, $h = \tilde{k}$ in the notation of Corollary \ref{cor average 2}.  Since $p\ge 1$, $k$ is subharmonic, and so by Corollary \ref{cor average 2} $h$ is also subharmonic and radially increasing.  In particular, there is some non-decreasing $g\colon[0,\infty)\to\R$ such that $h(x) = g(|x|)$.  So $S_r\,h(x) = g(r|x|) \le g(|x|) = h(x)$ for $r\in[0,1]$.  Integrating over $\R^n$ we have $\|f_r\|_{p,m}^p \le \int h(x)\,dx$ which equals $\|f\|_{p,m}^p$ by reversing the above argument.  This proves the result.

\medskip

If $0<p<1$, the above argument follows through as well since, if $f\in\lsh$ then $k=f^p$ is \lsh\ by Proposition \ref{properties}.  In particular, $k$ is non-negative and subharmonic, and so by Corollary \ref{cor average 2}, so is $\tilde{k}$.  The rest of the proof follows verbatim.

\end{proof}



We now show the strong hypercontractivity inequality for Gaussian measure and \lsh\ functions.  That is: $\|T_t f\|_{q,\gamma} \le \|f\|_{p,\gamma}$ whenever $f$ is \lsh\ and $t\ge t_J(p,q)$.  This is a generalization (from holomorphic functions to the much larger class of logarithmically-subharmonic functions) of Janson's original strong hypercontractivity theorem in \cite{j}.  Because our test functions $f$ are non-negative and the action of $T_t$ commutes with taking powers of $f$, this can be reduced to the following simplified form.

\begin{thm}\label{hcs}
Let $f$ be a  log--subharmonic function.  Then for every $r\in [0,1]$, one has
\begin{equation}\label{hc}
\|f_r  \| _{1/r^2,\g} \leq   \|f\|_{1,\g}. \end{equation}
\end{thm}
\begin{rem}
The inequality (\ref{hc}) means that the operators $S_r$ act as contractions between the spaces
\[  S_r:L^1_{ \lsh}(\g)\ra L^{1/r^2}_{  \lsh}(\g), \]
or, equivalently, the operator $T_t$ is a contraction between the  cones
\[  T_t:L^1_{  \lsh}(\g)\ra L^{ e^{2t}}_{  \lsh}(\g). \]
In fact, by Proposition \ref{properties}, one gets other hypercontractivity properties.  Applying the theorem to the function $f^p$, it follows that the operators $S_r$ are contractions
\[ S_r:L^p_{  \lsh}(\g)\ra L^{p/r^2}_{  \lsh}(\g), \]
and the operators $T_t$ are contractions
\[ T_t:L^p_{  \lsh}(\g)\ra L^{ e^{2t}p}_{  \lsh}(\g) \]
for any $p>0$.  Since $T_t$ is an $L^q$ contraction for any $q$ (Proposition \ref{contraction}), by the semigroup property the above implies that $T_t$ is a contraction from $L^p$ to $L^q$ for any $q\ge e^{2t}p$.  In other words, $T_t$ is a contraction from $L^p$ to $L^q$ provided that $t \ge \frac{1}{2}\log(q/p)$, the Janson time $t_J(p,q)$.  This is the strong hypercontractivity theorem proved in \cite{j} for holomorphic functions on $\CC^n\cong\RR^{2n}$; here we prove it for \lsh\ functions on $\RR^n$.
\end{rem}

\begin{proof}  
The case where $f=\log|g|$ with $g$ holomorphic on ${\CC^n}$ is implicitly proved in \cite{j} but is not given in this form. Using the ideas of Janson,  we will prove the general theorem.  Nelson's classical hypercontractivity result plays a crucial role here as in Janson's paper.  Let $P_t = e^{-tN}$ be the Ornstein--Uhlenbeck semigroup. Let us write it in the form 
\begin{equation} \label{eq OU kernel} P_tf(x)=\int M_r(x,y)f(y)\,\g(dy) \end{equation}
where $r=e^{-t}$ and $M_r$ is the {\em Mehler kernel}
\begin{equation} \label{eq Mehler}
M_r(x,y) = (1-r^2)^{-n/2}\,\exp\left(-\frac{r^2}{1-r^2}|x|^2 + \frac{2r}{1-r^2}\langle x,y\rangle -\frac{1+r^2}{1-r^2}|y|^2 \right).
\end{equation}
We can rewrite Equation \ref{eq OU kernel} in terms of Lebesgue measure as $P_tf(x) = \int K_r(x,y)f(y)\,dy$ where the modified kernel $K_r$ is given by
\[ K_r(x,y) = (1-r^2)^{-n/2}\,\exp\left(-\frac{|y-rx|^2}{1-r^2}\right). \]
Evidently $K_r(x,y)$ is constant in $y$ on spheres around $rx$. This implies that if $f\geq 0$ is subharmonic, then for all $t>0$ we have $P_t f(x)\geq f(e^{-t}x)$ (indeed, this is at the core of Janson's proof in \cite{j}). The classical hypercontractivity inequality of Nelson (cf.\ \cite{n}) is given by:  
\[  \| P_t f \|_{q(t), \g}\leq  \| f \|_{p,\g} \]
where $q(t)= (p-1)e^{2t}+1$ and $p>1$. Hence, for $f\ge 0$ subharmonic, we have Nelson's theorem for the dilation semigroup:
\begin{equation} \label{eq Nelson dilation} \|f(e^{-t}x) \|_{q(t),\g}\leq  \| f \|_{p,\g}. \end{equation}
Now take $f$ to be $\lsh$.  The function $f^{1/p}$ is also $\lsh$, so it is positive and subharmonic. Equation \ref{eq Nelson dilation} applied to $f^{1/p}$ becomes 
\[ \left(\int  f_{e^{-t}}(x)^{q(t)/p}d\g(x)\right)^{1/q(t)} \leq  \left(\int f(x) d\g(x))\right)^{1/p}. \]
This implies that
\[ \|f_{e^{-t}}\|_{q(t)/p ,\g} \leq  \|f\|_{1,\g}. \]
Observe that $\lim_{p\to \infty} \frac{q(t)}{p}=e^{2t}=\frac{1 }{r^2}$ where $r=e^{-t}$.   Applying Fatou's lemma, 
we obtain $\|f_r \|_{r^{-2},\g}\leq  \| f\|_{1,\g}$, the desired result.
\end{proof}

In the full hypercontractivity theory due to Nelson \cite{n}, $t_N(p,q)=\frac{1}{2}\log\frac{q-1}{p-1}$ is the smallest time to contraction, for all $L^p$-functions.  The analogous statement holds for Theorem \ref{hcs}; the exponent $1/r^2$ is optimal in this inequality (with Gaussian measure) over all \lsh\ functions.  In fact, it is optimal when restricted just to holomorphic functions on $\CC^n$, as is proved (in an analogous non-commutative setting) in \cite{k}; here we present a slightly different proof.

\begin{prop}\label{SharpGauss}
Let $r\in (0,1]$ and $C>0$. Assume that for some $p>0$, the following inequality holds for every \lsh\ function  $f$:
\begin{equation}\label{hcp}
 \|f_r\|_{p,\g} \leq  C\|f\|_{1,\g}.
\end{equation} 
Then $p\le 1/r^2$ and $C\geq 1$.
\end{prop}

\begin{rem} 
If $m$ is a probability measure then the $L^p$ norm $\|f\|_{p,m}$ is a non-decreasing function of $p$.  It follows that if Equation (\ref{hcp}) holds for a $p>1$ then it also holds for every $q\in[1,p)$.
\end{rem}

\begin{proof}
Consider the set of functions $f^a(x)=  e^{ax_1}$, which are all \lsh\ for $a>0$.  An  easy computation shows that 
$\|(f^a)_r\|_{p,\g} = \exp (r^2a ^2p/2)$; in particular, $\|(f^a)\|_{1,\g} = \exp(a^2/2)$.  The supposed inequality (\ref{hcp}) then implies that $\exp (r^2a ^2p/2)\leq C\exp (a ^2/2)$ for all $a>0$. Set $s= r^2p$.  Then $\exp(a^2(s-1)/2)\leq C$ 
for every real $a$.   Letting $a\to 0$ shows that $C\ge 1$; letting $a\to\infty$ shows that $s\le 1$.
\end{proof}
 
\begin{rem} \label{rem C>1} Hypercontractive inequalities very typically involve actual contractions (i.e.\ constant $C=1$ in Proposition \ref{SharpGauss}), since the time constant ($t_N$ or $t_J$ in this case) are usually independent of dimension, yielding an infinite-dimensional version of the inequality.  Indeed, in Nelson's original work \cite{n}, one main technique was to show that hypercontractivity held in all dimensions up to a fixed (dimension-independent) constant $C>1$.  The infinite-dimensional version then implies that $C=1$ is the best inequality, for if the best constant is $>1$ or $<1$, a tensor argument shows that in infinite dimensions the constant is $\infty$ or $0$, respectively.
\end{rem}

In the following, we will proceed along the lines of Remark \ref{rem C>1} and give a different proof of Theorem \ref{hcs}, with a non-optimal constant, that avoids direct use of Nelson's result, but produces a dimension-{\em dependent} constant.  First we need the following $L^{\infty}$  inequality.

\begin{lem}\label{infinity}  
Let $f$ be an \lsh\ function. Then for all $x\in\RR^n$,
\[ \exp (- \| x \|^2 /2) f(x)\leq \|f\|_{1,\g}. \]
\end{lem}

\begin{rem}
This inequality is sharp: take $f\equiv 1$. 
\end{rem}

\begin{proof}
From the proof of Proposition \ref{contraction} with $r=0$ and $m=\g$, it follows that for every non-negative subharmonic function $g$, the inequality $g(0)\leq \int g(x)\,d\g(x)$ holds.  Now take an \lsh\ function $f$  and $a,x\in{\RR}^n$.  It is easy to check that the translated function $y\mapsto f(x+y)$ is also \lsh. Then the function $f_x$ given by
\[ f_x(y)= f(x+y)e^{\langle a,y \rangle} \]
is a product of two \lsh\ functions, and so is \lsh\ by Proposition \ref{properties}.  In particular, it is non-negative and subharmonic.   Applying the  last inequality to $f_x$, we get 
\begin{equation} \label{eq Linfty 1} f(x)= f_x (0)\leq \int f(x+y)  e^{\langle a,y\rangle}d\g(y). \end{equation}
Make the change of variables $v=x+y$.  Then the right-hand-side of Equation \ref{eq Linfty 1} becomes 
\[ \int f(v)   e^{\langle a,v-x\rangle} e^{-|v-x|^2/2}\,(2\pi)^{-n/2}\,dv \]
and is equal to
\[ e^{-\langle a,x \rangle }e^{ - \| x \|^2/2} \int f(v)  e^{- \|v \| ^2/2}  e^{\langle a,v \rangle} e^{\langle x,v\rangle }  (2\pi)^{-n/2}\,dv. \]
The conclusion follows by taking $a=-x$. 
\end{proof}

\begin{prop} (Hyperboundedness) For  the constant  $C=e^{1/2e}>1$, and for every $r\in [0,1]$, the following inequality  is true for any \lsh\ function $f$  on ${\RR}^n$:
\[ \| f_r \| _{1/r^2,\g} \le C^n\, \| f  \| _{1,\g}. \]
\end{prop}

\begin{proof}
Denote  $I_r= \left(\|f_r\|_{1/r^2,\g}\right)^{1/r^2} =  \int f(rx)^{1/r^2}\, d\g(x)$.  By the change of variables $y=rx$, the integral  $I_r$ can be written as
\[ \frac{(2\pi)^{-n/2}}{r^n}\int a(y)^{1/r^2}\, dy, \] 
where $a(y)= f(y) \exp(-|y|^2/2)$.  By Lemma \ref{infinity},  $a(y)\leq I_1$, which implies that
$a(y)^{1/r^2-1}\leq I_1^{1/r^2-1}$. 

\medskip

\noindent Now write  $a(y)^{1/r^2}= a(y) a(y)^{1/r^2-1}$.  Then
\[ I_r=\frac{(2\pi)^{-n/2}}{r^n} \int a(y) a(y)^{1/r^2-1}dy
\leq r^{-n} \left(\int a(y) (2\pi)^{-n/2}\, dy\right) I_1^{1/r^2-1} 
= r^{-n} I_1^{1/r^2-1}I_1 = r^{-n} I_1^{1/r^2}. \]
Consequently, $(I_r)^{r^2}\le (r^{-nr^2})\,I_1$.  This can be read as: $ \| f_r \|_{1/r^2,\g} \leq r^{-nr^2} \|f\|_{1,\g}$.  The function $r\mapsto ({\frac{1}{r}})^{r^2}$ is maximized on $(0,1]$ by $C=e ^{\frac{1}{2e}}\approx 1.445$. This completes the proof.
\end{proof}

\section{Hypercontractivity inequalities for probability measures} \label{sect convolution}

In this section we study hypercontractivity properties of \lsh\ functions with respect to any probability measure $m$.  We have already seen in Proposition \ref{contraction}
that, for rotationally invariant measures $m$, the semigroup $S_r$ is always an $L^p$ contraction.

\begin{thm}\label{convol}  Fix $q>1$ and $r\in(0,1]$.  Suppose that $\mu_1$ and $\mu_2$ are two probability measures on $\RR^n$ which verify the hypercontractivity
inequality
\begin{equation}\label{hcq}
\| f_r \|_{q,\mu}
\leq   \|f\|_{1,\mu}
\end{equation}
for any continuous \lsh\ function $f$.  It at least one of $\mu_1$ and $\mu_2$ is compactly-supported, then the convolved measure $\mu_1*\mu_2$  also satisfies (\ref{hcq}).  
\end{thm}

\begin{proof}
Let $f$ be a continuous \lsh\ function, and suppose $\mu_1$ is compactly-supported.  We have
\begin{eqnarray*}
\int f(rz)^q d(\mu_1*\mu_2 )(z)&=&
\int\int f(rx+ry)^q\,d\mu_1(x)\,d\mu_2(y)\\
&\leq& \int \left(\int f(x+ry)\,d\mu_1(x) \right)^q d\mu_2(y) 
\end{eqnarray*}
since the function $x\mapsto f(x+ry)$ is continuous \lsh\ for each fixed $y\in\R^n$, and $\mu_1$ satisfies (\ref{hcq}).  Let $h(y)=\int f(x+y)\,d\mu_1(x)$, so that we have proven that
\begin{equation} \label{eq convol 1} \|f_r\|_{q,\mu_1\ast\mu_2}^q \le \int h(ry)^q\,d\mu_2(y) = \|h_r\|_{1,\mu_2}^q. \end{equation}
Since $f$ is continuous, the function $(x,y)\mapsto f(x+y)$ is continuous in both variables, and also \lsh\ in each.  Since $\mathrm{supp}\,\mu_1$ is compact and $f$ is continuous, all the conditions of Corollary \ref{cor average 1} are satisfied, and so $h$ is \lsh.  Thence, by the assumption of the theorem, the quantity on the right-hand-side of Equation \ref{eq convol 1} is bounded above by $\|h\|_{1,\mu_2}^q$.  By definition,
\[ \|h\|_{1,\mu_2} = \int h(y)\,d\mu_2(y) = \int\int f(x+y)\,d\mu_1(x)\,d\mu_2(y) = \|f\|_{1,\mu_1\ast\mu_2}, \]
and this proves that Inequality \ref{hcq} also holds for $\mu_1\ast\mu_2$.  \end{proof}


Most of the following results of this section concern the 1--dimensional case, i.e.\ log--convex functions on the real line.  In that case, one has the following surprisingly general hypercontractivity inequality. 

\begin{prop}\label{logconv}
For every symmetric probability measure $m$ on $\RR$, and  for any logarithmically convex function $f$ on $\RR$, the following inequality is true for any $r\in(0,1]$:
$$ \| f_r \| _{1/r,m}
\leq  \| f \|_{1,m}.$$   
\end{prop}
\begin{rem} Translating this statement into additive language, the dilation semigroup $T_t$ satisfies strong hypercontractivity with time to contraction at most $2\cdot t_J$, for any symmetric probability measure on $\RR$, for log--convex functions.  As explained above, a simple scaling $f\mapsto f^p$ yields the comparable result from $L^p\to L^q$ for $q\ge p >0$. \end{rem}

\begin{proof}
By the log--convexity of $f$, for any $x\in\RR$
\[ f(rx)\leq f(0)^{1-r} f(x)^r, \]
which implies that $f(rx)^{1/r}\leq f(0)^{1/r-1} f(x)$.  Then by $m$-integration,
\[ \int f(rx)^{1/r}dm(x)\leq f(0)^{1/r-1}\vert\vert f\vert\vert _{1,m}. \] 
Since $f$ is convex, $f(0)\le \frac{1}{2}[f(x)+f(-x)]$ for all $x$.  Integrating and using the symmetry of $m$ yields $f(0)\leq \vert\vert f\vert\vert _{1,m}$.   Consequently,
\[ \int f(rx)^{1/r}dm(x)\leq \vert\vert f\vert\vert _{1,m}^{1/r}, \] 
and the Proposition follows.
\end{proof}

\begin{rem} Proposition \ref{logconv} remains true for rotationally invariant
measures $m$ and log--convex functions $f$ on $\RR^n$.
This proof fails, however, for general \lsh\ functions on $\RR^n$ when $n\geq 2$.
\end{rem}

\begin{rem} Subject to additional regularity on $m$, the symmetry condition in Proposition \ref{logconv} can be replaced with the much weaker assumption that $m$ is centred: i.e.\ $m$ has a finite first moment, and $\int x\,m(dx)=0$.  In short, fix a log-convex $f$, and suppose that $m$ is regular enough that the function $\eta(r) = \int f(rx)\,m(dx)$ is differentiable, so that $\eta'(r) = \int f'(rx)x\,m(dx)$.  (It is easy to see, from convexity of $f$, that $f_r\in L^1(m)$ for each $r$, provided $f\in L^1(m)$.)  Then $\eta'(0) = f'(0) \int x\,m(dx)=0$, and since $f$ is convex, $f'$ is increasing which means that $x f'(rx) \ge x f'(x)$ for all $x$, $r\ge0$, so $\eta'(r) \ge \eta'(0) = 0$.  Thus, $\int f\,dm = \eta(1) \ge \eta(0) = f(0)$, and the rest of the above proof follows.  For this to work, it is necessary to assume (at minimum) that the functions $\frac{\partial}{\partial r} f(rx) = f'(rx)x$ are uniformly bounded in $L^1(m)$; a convenient way to achieve this is to assume that functions $g\in L^1(m)$ for which $x\mapsto x g'(x)$ is also in $L^1(m)$ are dense in $L^1(m)$.  The kinds of measures for which such a Sobolev-space density is known is a main topic of our subsequent paper \cite{kl}.
\end{rem}

The problem in general is to find, for a fixed measure $m$, the maximal exponent $q$ such that $\| f_r \|_{q,m} \leq   \|f\|_{1,m}$ for every $r\in(0,1]$ and any log-convex function $f$ on $\R$.  For symmetric Bernoulli measures we will show that  the optimal exponent $q$ is the same as for Gaussian measures. 

\begin{prop}\label{Bern} 
If $m=\frac{1}{2}(\delta _1 +\delta _{-1})$ then
\begin{equation}\label{bern} \| f_r \|_{1/r^2,m}\leq \|f\|_{1,m} \end{equation}
for every $r\in(0,1]$ and any log-convex function $f$. 
\end{prop}

\begin{rem} \label{remark Bern} It follows from Proposition \ref{Bern}, and a simple rescaling argument, that the same strong hypercontractivity inequality holds for any symmetric Bernoulli measure $\frac12(\delta_a + \delta_{-a})$, $a>0$.
\end{rem}

\begin{proof}
{\it Step 1.} We justify that it is sufficient  to prove the proposition for the two-parameter family of functions $h(x)= C\exp(ax)$ with $a\in\R$ and $C>0$.  Take $f$ strictly positive. Then there exists $h$ of the form $C\exp(ax)$ such that  
the  functions $f$ and $h$ are equal on the set $\{-1,+1\}$.   Assume now that $f$ is log-convex. 
Then $f\leq h$ on $[-1,1]$, and in particular  
 $f(r)\leq h(r)$ and  $f(-r)\leq h(-r)$. This implies that 
$$\int f(rx)^{1/r^2}\,dm(x)
\leq \int h(rx)^{1/r^2}\,dm(x).$$
If the function $h$ satisfies (\ref{bern}), we obtain
\[ \|f_r\|_{q,m}\leq \|h_r\|_{q,m}\leq \|h\|_{1,m}=\|f\|_{1,m},\]
the last equality following from the fact that $f$ and $h$ coincide on the support of $m$.  This gives the inequality (\ref{bern}) for $f$. 

\medskip

{\it Step 2}: We show the inequality (\ref{bern}) for $f(x)=e^{ax}$ (the constant $C$ obviously factors out of the desired inequality).  This is essentially  an exercise. One has to prove that 
\[ \left(\int \exp(ax/r)dm(x)\right)^{r^2} \leq \int \exp(ax) dm(x), \] 
i.e.\ $\left(\cosh(\frac{a}{r})\right)^{r^2} \leq \cosh a$  for $a$ real and $r\in (0,1]$.  Put $s=1/r$. Then $s\geq 1$ and the required  inequality becomes $\cosh(s a)\leq (\cosh a)^{s^2}$.   Taking logarithms and next  dividing by $s^2a ^2$, we are left to prove that 
\[ \frac{\log (\cosh (s a))}{s^2 a^2} \leq \frac{\log (\cosh a)}{a^2}. \] 
In other words, we must prove that the function $\log (\cosh x)/x ^2$ is decreasing  for $x\geq 0$. Taking the derivative, it is sufficient to see that  $\rho(x) = x\tanh x -2\log (\cosh x)$ is nonpositive for $x\geq 0$.  Well, $\rho(0)=0$, and $\rho'(x) = x/\cosh ^2 x-\tanh x=\frac{x-\sinh x\cosh x}{\cosh^2 x}$. This last  quotient is non-positive for its numerator is  equal to $x-(\sinh 2x)/2$.
\end{proof} 

\begin{rem}
Proposition \ref{bern} could be obtained from an inequality of A.\ Bonami \cite{b}  similarly to the manner in which Theorem \ref{hcs} was obtained from Nelson's hypercontractivity theorem for Gaussian measures.  She proved that for symmetric Bernoulli measures 
 the same classical hypercontractivity inequalities as for the Gaussian measure hold.
In order to prove  Proposition \ref{bern} for a log-convex function $f$, one compares it to the affine function which takes the same value as $f$ on $\{-1,1\}$. For a function  on $\{-1,1\}$, there is a unique affine function on the line which extends it. Thus  one can identify the space $C\{-1,1\}$ of functions on $\{-1,1\}$ and the space of affine functions on the line. We omit the details.  
\end{rem}

\begin{cor}
The symmetric uniform probability measure $\lambda_a$ on $[-a,a]$, $a>0$, satisfies the strong hypercontractivity property $\|f_r\|_{1/r^2,\lambda_a}\le \|f\|_{1,\lambda_a}$ for all \lsh\ functions.
\end{cor}

\begin{proof}
Let $ m_x=\frac{1}{2}(\delta_x+\delta_{-x})$.
It is easy to see that
\[ \mu_k:=m_{\frac{1}{2}}*m_{\frac{1}{4}} *\ldots* m_{\frac{1}{2^k}} \rightharpoonup \lambda_1,\ \ \ \ \ k\to\infty, \]
where we denote by $\rightharpoonup$ the convergence in law. By the Proposition \ref{Bern} (and the proceeding Remark \ref{remark Bern}) and Theorem \ref{convol}, the inequality (\ref{bern}) holds for the  measures $\mu_k$. The supports of the measures $\mu_k$ and $\lambda_1$ are compact and included in the segment $[-1,1]$. If $f$ is log--convex on $\RR$, it is continuous and the convergence $\int_{-1}^1 f\,d\mu_k \ra \int_{-1}^1 f\,d\lambda_1$
follows from the convergence in law $\mu_k \Rightarrow \lambda_1$.  The statement for all $a>0$ now follows from a simple rescaling argument.
\end{proof}

\section{Logarithmic Sobolev Inequalities for  LSH functions on $\RR$} \label{sect Log Sobolev}

In this section we will prove that a strong log--Sobolev inequality holds for log-subharmonic functions and Gaussian measures in 1 dimension. We will also show log-Sobolev Inequalities for  other 1--dimensional measures from previous sections, for  which we showed the strong hypercontractivity for \lsh\ functions (symmetric Bernoulli measures, uniform symmetric measures or any symmetric probability measure on $\RR$.)  Considerably more general log-Sobolev inequalities (in all dimensions) hold in the \lsh\ category; this will be discussed in \cite{kl}.

\medskip

Recall that the classical Gaussian Logarithmic Sobolev Inequality, cf. \cite{a,g4}, is

\begin{equation}\label{LSIclassic}
 {\mathcal E}(f^2)=\int |f|^2\log |f|^2 d\g -\|f\|_{2,\g}^2 \log\|f\|_{2,\g}^2 \le 2 \int f Lf d\g
\end{equation}
where $\g$ is the standard Gaussian measure, $L=-\Delta+E$ is the generator of the Ornstein--Uhlenbeck semigroup and $f\in{\mathcal A}$, a standard algebra contained in the domain of the operator $L$.  For the Ornstein--Uhlenbeck semigroup ${\mathcal A}$ can be chosen as the space of ${\mathcal C}^\infty$ functions with slowly increasing derivatives. The expression ${\mathcal E}(f)$ is often called the entropy of $f$.

\medskip

The celebrated theorem of Gross \cite{g4} establishes the equivalence between the  hypercontractivity property of a semigroup $T_t$ with invariant measure $\mu$ and the log--Sobolev inequality relative to the generator $L$ of $T_t$.  More precisely, recalling the Nelson time $t_N=\frac{1}{2}\ln \frac{q-1}{p-1}$, the hypercontractivity inequalities $\|T_t f\|_{q,\mu}\le \|f\|_{p,\mu}$
for $t\ge c\,t_N(p,q)$ for $1<p\le q<\infty$ are, together, equivalent to the single log--Sobolev Inequality
 \begin{equation}\label{classicalLSIc}
{\mathcal E}(f^2)=\int |f|^2\log |f|^2 d\mu -\|f\|_{2,\mu}^2 \log\|f\|_{2,\mu}^2 \le 2c \int f Lf\, d\mu.  \end{equation}
In the Gaussian case these inequalities indeed hold with $c=1$.

\medskip

Now, let $f$ be a positive  subharmonic  function of class ${\mathcal C}^2$. Then  $\Delta f  \ge 0$
and $Lf\le Ef$.  From (\ref{LSIclassic}) it follows that 
\begin{equation}\label{LSIlshL2}
{\mathcal E}(f^2)= \int |f|^2\log |f|^2\, d\g -\|f\|_{2,\g}^2 \log\|f\|_{2,\g}^2 \le
2 \int f Ef\, d\g.
\end{equation}
If, moreover, $f$ is \lsh,  we set $g=f^2$ and using the fact that $Eg=2fEf$
we can write the last inequality as
\begin{equation}\label{LSIlshL1}
 {\mathcal E}(g)=\int  g\log g\, d\g -\|g\|_{1,\g} \log\|g\|_{1,\g} \le
 \int  Eg\, d\g.
\end{equation}

In this section we will prove  that a {\bf stronger} Log-Sobolev Inequality 
\begin{equation}\label{LSIlshL2better} 
{\mathcal E}(g) \le \frac{1}{2} \int Eg\,d\gamma
\end{equation}
holds for log--subharmonic functions $g$ and Gaussian measure $\gamma$ in 1 dimension, as well as with $\gamma$ replaced by a symmetric Bernoulli measure or symmetric uniform measure on $\RR$. Indeed, the constant factor $1$ from the inequality (\ref{LSIlshL1}) is optimal in general; here we prove that in the \lsh\ category, the constant is instead $\frac12$ (as in  \ref{LSIlshL2better}).

It may seem surprising that the integrals  $\int f Ef\, d\g$ from (\ref{LSIlshL2})
and, equivalently, $\int  Eg\, d\g$ from (\ref{LSIlshL1}) are positive when $f$ and $g$
are \lsh\ functions. The following proposition explains this phenomenon, which holds more generally for subharmonic functions.
\begin{prop}
Let  $m$ be a probability measure on  ${\mathbb R}^n$
which is $O(n)$ invariant, and let and $g\in{\mathcal C}^1$ be a subhamronic function.
Then
\[ I=\int Eg(x) dm(x)\ge 0. \]
\end{prop}

\begin{proof}
We have
\[ I=\int dm(x) \int_{O(n)} Eg(\alpha x)\,d\alpha, \]
where $d\alpha$ denotes the Haar measure on $O(n)$.  Denote by $\sigma$ the normalized Lebegue measure on the unit sphere $S^{n-1}$. If $r=\|x\|$, we have
 \begin{eqnarray*}
\int_{O(n)} Eg(\alpha x)\,d\alpha=\int_{S^{n-1}} (Eg)(ru)\,\sigma(du) &=&
r \int_{S^{n-1}}\frac{\partial g }{\partial r} (ru)\, \sigma(du) \\
&=&r\frac{\partial }{\partial r}\int_{S^{n-1}} g(ru)\,\sigma(du) \ge 0
 \end{eqnarray*}
because the function $r\mapsto \int_{S^{n-1}} g(ru)\,\sigma(du)$
is  increasing (cf. Corollary \ref{cor average 2}).
\end{proof}

\subsection{Log-Sobolev Inequalities for measures with compact support} 

The following techniques work, in principle, quite generally.  However, the usual approximation techniques to guarantee integrability (convolution approximations and cut-offs) are unavailable in the category of subharmonic functions.  As such, we include this section which develops the relevant log-Sobolev inequalities in all dimensions, but only for compactly--supported measures (i.e.\ do the cut-off in the measure rather than the test functions).  Extension of these results to a much larger class of measures is the topic of \cite{kl}.

\begin{thm}\label{LSItheorem}
 Let $\mu$ be a probability measure on ${\mathbb R}^n$ with 
{\bf compact support}. Suppose  that for some $c>0$, the following strong hypercontractivity property holds: for $0<p\le q<\infty$ and  $f\in  L^p_{\lsh}(\mu)$,
\[ \|f_{e^{-t}}\|_{q,\mu}\le \|f\|_{p,\mu}\ \ \ {\rm for} \ t\ge c\cdot \textstyle{\frac{1}{2}}\log\frac{q}{p}. \]
Then for  any log--subharmonic function $f\in {\mathcal C^1}$ the following logarithmic Sobolev inequality holds:
 \begin{equation}\label{LSIcompact}
\int f^2\log f^2 d\mu -\|f\|_{2,\mu}^2 \log\|f\|_{2,\mu}^2 \le c \int f Ef d\mu.
 \end{equation}
\end{thm}
\begin{rem}
\begin{enumerate}
\item The condition $f\in {\mathcal C^1}$ is natural to ensure a good sense
of the expression $Ef$ in (\ref{LSIcompact}). In the  classical case in \cite{a} one 
supposes $f\in{\mathcal A}\subset{\mathcal C^\infty}$ and such an LSI inequality is equivalent  to the hypercontractivity property (\cite{a}, Theorem 2.8.2).
\item  In the case of strong hypercontractivity with optimal $q=p/r^2$ (symmetric Bernoulli measures and their convolutions, symmetric uniform measures on $[-a,a]$), the constant $c$ is equal to $1$. Also Gaussian measures on $\RR^n$ have the constant $c=1$ but evidently they are not covered by the Theorem \ref{LSItheorem}.
When  $q=p/r$ (any symmetric measure on ${\mathbb R}$), the constant $c$ is equal to $2$.  The time $t_J=\frac{1}{2}\log\frac{q}{p}$ appearing in Theorem \ref{LSItheorem} is Janson's time.
\item Theorem \ref{LSItheorem} is stated and proved here for compactly-supported measures, a class not including the most important Gaussian measures.  In the next section we prove it {\em does} hold for Gaussian measure $1$ dimension, see Theorems \ref{LSItheoremGaussianL2} and \ref{LSItheoremGaussianL1}).  In fact the same strong log-Sobolev inequality holds for Gaussian measures (and beyond) in all dimensions; this will be covered in \cite{kl}.  Let us reiterate that the following proof applies to a much wider class of measures, but the precise regularity conditions are complicated by the fact that cut-off approximations do not preserve the cone of log--subharmonic functions. 
\end{enumerate}
\end{rem}

\begin{proof}
Let $p=2$ and $t$ be the critical time $t=c\cdot \frac{1}{2}\log\frac{q}{p}$.
Then the variable $r=e^{-t}$ satisfies $q(r)=2r^{-2/c}$.
 The method of proof is classical and consists of differentiating the function
$$
\alpha(r)=\|f_r\|_{q(r),\mu}
$$
 at $r=1$. By strong hypercontractivity, $\alpha(r)\le \alpha(1)$, so $\alpha'(1)\ge 0$ if  we prove 
the existence of this derivative.

\medskip

Define $\beta(r)=\alpha(r)^{q(r)}=\int f(rx)^{q(r)} d\mu(x)$ and let 
$\beta_x(r)=f(rx)^{q(r)}$, so that
$\beta(r)=\int \beta_x(r) d\mu(x)$. Then
$$
\frac{\partial}{\partial r} \log\beta_x(r)=q'(r) \log f(rx) +\frac{q(r)}{f(rx)}x\cdot \nabla f(rx).
$$
Since $q'(r)=-\frac{2}{rc}q(r)$, we  compute
\begin{equation}\label{derivative}
 \beta'_x(r)=-\frac{2}{rc} f_r(x)^{q(r)} \log f_r(x)^{q(r)} +\frac{q(r)}{r}f_r(x)^{q(r)-1} (Ef)_r(x).
\end{equation}
 Let $0<\epsilon<1$. As $f\in {\mathcal C^1}$, the expression on the right-hand side of (\ref{derivative})
is bounded for $r\in (1-\epsilon, 1]$ and $x\in {\rm supp}(\mu)$ (which is compact).  The Dominated Convergence Theorem then implies that
\begin{equation}\label{exchange}
\beta'(r)=\frac{\partial  }{ \partial r} \int \beta_x(r)\, d\mu(x)= \int \beta'_x(r)d\mu(x).
\end{equation}
Finally, since $\alpha(r)=\beta(r)^{1/q(r)}$ and $\beta>0$, we have that $\alpha$ is 
$ {\mathcal C^1}$ on $(1-\epsilon,1]$ and a simple calculation shows that
\[ \alpha'(r)=\frac{\alpha(r)}{q(r) \beta(r)}\left[\frac{2}{rc} \beta(r)\log\beta(r) +\beta'(r)\right].  \]
Now, taking $r=1$, applying  $\alpha'(1)\ge 0$ and the formulas (\ref{derivative})
and (\ref{exchange}) we obtain 
 \begin{eqnarray*}
0&\le& \frac{2}{c} \beta(1)\log \beta(1) +\beta'(1)\\
 &=& \frac{2}{c} \|f\|_{2,\mu}^2 \log\|f\|_{2,\mu}^2 - \frac{2}{c} \int f^2\log f^2 d\mu + 2 \int f Ef d\mu,
 \end{eqnarray*}
and this is the logarithmic Sobolev inequality (\ref{LSIcompact}).
\end{proof}

\noindent For $p>0$ we define spaces $L^p_E(\mu)=\{f\,;\, f\in L^p(\mu)\ {\rm and}\ Ef\in L^p(\mu)\}$
and $L^p(\mu) \log L^p(\mu)=\{f\,;\,\int f^p|\log f^p| d\mu <\infty\}$.  The former is a Sobolev space, the latter an Orlicz space, related to the logarithmic Sobolev inequality \ref{LSIcompact}; indeed, in the case $p=2$, they are the spaces for which the right-- and left--hand sides (respectively) of that inequality are finite.

\medskip

Appealing to the surprising Proposition \ref{logconv},  and Theorem \ref{LSItheorem}, we have the following.

\begin{cor}\label{LSItheoremFull} Let $\mu$ be a symmetric probability measure on ${\mathbb R}$.  Then for  any log-subharmonic function $f\in L^2(\mu) \log L^2(\mu) \cap L^2_E(\mu)\cap{\mathcal C^1}$ the  following logarithmic Sobolev inequality  holds:
 \begin{equation*} 
\int f^2\log f^2 d\mu -\|f\|_{L^2(\mu)}^2 \log\|f\|_{L^2(\mu)}^2 \le
2 \int f Ef d\mu.
 \end{equation*}
\end{cor}

\begin{rem} In the classical case it is sufficient to suppose only $f\in L^2_E(\mu)$; this actually implies that $f\in  L^2(\mu) \log L^2(\mu)$.  The proof of this fact involves  approximation by more regular (e.g. compactly  supported or bounded) functions, and these tools are unavailable to us here.
\end{rem}

\begin{proof}
By Proposition \ref{logconv} the measure $\mu$ as well as the measures
$\mu_N=\mu|_{[-N,N]} + \mu([-N,N]^c)\delta_0$ verify the strong hypercontractivity
property for  \lsh\  functions with $q=p/r$ and $c=2$.  Let $f$ verify the hypothesis of the corollary, and set $f^\e = f+\e$; it is easy to check that $f^\e$ also verifies all the conditions of the corollary.  By Theorem  \ref{LSItheorem},
for each $N$
\begin{equation*} 
\int (f^\e)^2\log (f^\e)^2 d\mu_N -\|f^\e\|_{2,\mu_N}^2 \log\|f^\e\|_{2,\mu_N}^2 \le
2 \int f^\e Ef^\e d\mu_N.
 \end{equation*}
When $N\ra \infty$, $\mu_N\rightharpoonup\mu$ (weak convergence), and since $f^\e\in\mathcal{C}^1$ and is strictly positive, all the functions $(f^\e)^2$, $(f^\e)^2\log (f^\e)^2$, and $f^\e Ef^\e$ are continuous; hence the integrals in the last formula converge to analogous integrals in terms of $f^\e$ with respect to the measure $\mu$.  Finally, we can let $\e\downarrow 0$ to achieve the result, by the Monotone Convergence Theorem.
\end{proof}

{\bf EDITED UP TO HERE}

\begin{cor}\label{LSItheoremFullL1}
 Let $\mu$ be a symmetric probability measure on ${\mathbb R}$.  
Then for  any log--subharmonic function $f\in L^1(\mu) \log L^1(\mu) \cap L^1_E(\mu)\cap{\mathcal C^1}$ the  following logarithmic Sobolev inequality  holds:
 \begin{equation*} 
\int f\log f d\mu -\|f\|_{1,\mu} \log\|f\|_{1,\mu} \le
 \int  Ef d\mu.
 \end{equation*}
\end{cor}
\begin{proof}
 The proof is similar to the proof of the Corollary \ref{LSItheoremFull}.
Note, nevertheless, that Corollary  \ref{LSItheoremFullL1} does not follow
from Corollary \ref{LSItheoremFull} because the  hypothesis $Ef\in L^1(\mu)$
is weaker  than the condition $Ef\in L^2(\mu)$ supposed in Corollary \ref{LSItheoremFull}
(all other integrability hypotheses are equivalent by the transformation $f \mapsto f^2$ which maps $L^2$ onto $L^1$).
\end{proof}

\subsection{Log-Sobolev Inequality for Gaussian measures on $\RR$}

We formulate two versions of the Logarithmic Sobolev Ineaquality for log-subharmonic functions:
in the classical context $L^2(\gamma)$ (Theorem \ref{LSItheoremGaussianL2})
and in the more natural and technically simpler case $L^1(\gamma)$.

Both cases are nearly equivalent since $f\in L^2(\gamma)$ and log--subharmonic
is equivalent to $f^2\in L^1(\gamma)$ and log--subharmonic. But the integration
hypotheses of the theorems are slightly different, cf. the discussion in the proof
of the Corollary \ref{LSItheoremFullL1}.

\begin{thm}\label{LSItheoremGaussianL2}
 Let $\gamma$ be  the Gaussian measure with  density $\frac{1}{\sqrt{2\pi}}e^{-x^2/2}$ on ${\mathbb R}$.  
Then for any \lsh\ and ${\mathcal C^1}$ function  $f\in L^2(\gamma ) \log L^2(\gamma  ) \cap L^2_E(\gamma)$ the  following logarithmic Sobolev inequality  holds
 \begin{equation}\label{LSIGaussL2} 
\int f^2\log f^2 d\g -\|f\|_{2,\g}^2 \log\|f\|_{2,\g}^2 \le
 \int f Ef d\g.
 \end{equation}
\end{thm}

\begin{thm}\label{LSItheoremGaussianL1}
 Let $\gamma$ be  as in Theorem \ref{LSItheoremGaussianL2}. Then for  any \lsh\ and ${\mathcal C^1}$ function   $f\in L^1(\gamma ) \log L^1(\gamma  ) \cap\ L^1_E(\gamma)$ the  following logarithmic Sobolev inequality  holds
 \begin{equation}\label{LSIGaussL1} 
\int f\log f d\gamma -\|f\|_{1,\gamma} \log\|f\|_{1,\gamma} \le
 \frac{1}{2}\int Ef d\gamma.
 \end{equation}
\end{thm}

 Note that the method of the proof of Corollary \ref{LSItheoremFull} cannot be applied
because we do not know if the measures $\gamma_N$ have the strong hypercontractivity 
property with Gaussian constant $c=1$; by the  Theorem \ref{logconv} they have it
with $c=2$ and we would obtain the weaker inequality (\ref{LSIlshL2}).
Instead, we will use the Proposition \ref{Bern}, the Theorem \ref{convol}
and some   results about   strengthened versions  of the DeMoivre--Laplace Central Limit Theorem, proved in the following subsection.  This approach mirrors, to some extent, Gross's proof of the Gaussian log-Sobolev inequality in \cite{g4}.

\begin{rem}
For the log--subharmonic  functions $f(x)=e^{ax}, a>0$ there is equality
in (\ref{LSIGaussL2}) and (\ref{LSIGaussL1}). Thus
 the constant $c=1$ is optimal in (\ref{LSIGaussL2}) and  
the constant $\frac{1}{2}$ is optimal in  (\ref{LSIGaussL1}). 
\end{rem}

 \subsection{Strengthened DeMoivre--Laplace Central Limit Theorems }\label{CLT}
 
\begin{thm}\label{CLTepsilon}
  Let $X_n$ be independent, identically distributed Bernoulli random variables
with $P(X_k=0)=P(X_k=1)=\frac{1}{2}.$ Let 
$$
S_n=\frac{X_1+\ldots+X_n - \frac{n}{2}}{\frac{\sqrt{n}}{2}}
$$
and let $Y$ be an  $N(0,1)$ random variable. 
Then for every continuous  function $f$ integrable with respect to the normal $N(0,1)$ law
$\gamma$ and  such that   $|f(x)|\le e^{(\frac{1}{2}-\epsilon)x^2}$ for some $0<\epsilon< 1/2$, we have
$$
\lim_n \EE f(S_n) =\EE f(Y).
$$ 
 \end{thm}
\begin{proof}
Let $Y_k=2X_k-1$.  We have $
S_n=\sum_{k=1}^n\frac{1}{\sqrt{n}} Y_k.
$
The independent  random variables $Y_k$  take the values $1$ or $-1$ with
probability $1/2$ and the  Hoeffding inequality (see e.g.\ \cite{dud}, Prop.\ 1.3.5) implies that
 \begin{equation}\label{hoeff}
P(S_n>u)\le e^{-u^2/2}.
 \end{equation}
Let $0<\epsilon<\frac{1}{2}$ and $F_\epsilon(x)=\exp\{(\frac{1}{2}-\epsilon) x^2\}$.
It follows from (\ref{hoeff}) that
\begin{equation}\label{Fepsilon}
 \EE F_\epsilon(S_n)\ \  \ra \ \  \EE F_\epsilon(Y).
\end{equation}
Indeed, since $P(F_\epsilon(S_n)>x)=1$ for $x\le 1$,
\[ \EE F_\epsilon(S_n) =\int_0^\infty P(F_\epsilon(S_n)>x) dx= 1 + \int_1^\infty P(F_\epsilon(S_n)>x) dx. \]
In the last  integral, change the variables $F_\epsilon(t)=x$. We obtain
$$
\int_1^\infty P(F_\epsilon(S_n)>x) dx= \int_0^\infty F_\epsilon'(t) P( S_n>t)dt=
(1-2\epsilon)\int_0^\infty t F_\epsilon(t) P( S_n>t)dt.
$$
By the Central Limit Theorem we have $ \lim_{n\ra \infty} P( S_n>t)= P(Y>t)$.
Using (\ref{hoeff}) and the Dominated Convergence Theorem
we see that 
$$\int_1^\infty P(F_\epsilon(S_n)>x) dx\ \ \ra \ \  \int_1^\infty P(F_\epsilon(Y)>x) dx
$$ 
and we conclude that (\ref{Fepsilon}) is true.

\medskip

Now, let $f$ be continuous and  $0\le f \le F_\epsilon$ for a fixed $\epsilon$. Take $N>0$. Decompose $\EE(F_\epsilon(S_n))=\EE(F_\epsilon(S_n){\bf 1}_{\{|S_n|\le N\}})+ 
\EE(F_\epsilon(S_n){\bf 1}_{\{|S_n| > N\}})$. The Central Limit Theorem  implies that
$$
\EE(F_\epsilon(S_n){\bf 1}_{\{|S_n|\le N\}})\ \ \ra \ \ \EE(F_\epsilon(Y){\bf 1}_{\{|Y|\le N\}}).
$$
Thus (\ref{Fepsilon}) and the integrability of $F_\epsilon$ with respect to the Gaussian law of $Y$
imply that 
$$
\forall \delta>0 \ \exists N>0 \ \forall n\ \ \ \EE(F_\epsilon(S_n){\bf 1}_{\{|S_n|> N\}}) <\delta.
$$
As $0\le f\le F_\epsilon$, we have
$$
\forall \delta>0 \ \exists N>0 \ \forall n\ \ \ \EE( f(S_n){\bf 1}_{\{|S_n|> N\}}) <\delta.
$$
By  the Central Limit Theorem, for every $N>0$ we have  
$\EE( f(S_n){\bf 1}_{\{|S_n|\le N\}}) \to \EE(f(Y) {\bf 1}_{\{|Y|\le N\}})
$
and it follows that $\EE( f(S_n)) \to \EE( f(Y))$.
\end{proof}

\noindent In the sequel we denote by $\mu_n$ the law of $S_n$.  Denote $\Psi(x)=P(Y>x)$ the tail function of the Gaussian distribution $\gamma$ and $\Psi_n(x)=P(S_n>x)$ the tails of the random variables $S_n$.

\begin{prop}\label{g'Psi}
 If $g\in L^1(\gamma)$ is in ${\mathcal C}^1([0,\infty))$  and
  $g$ is {\bf strictly increasing} on $[x_0,\infty)$ for an $x_0\ge 0$\\
then
\begin{equation}\label{g'gauss}
 \int_{x_0}^\infty g\, d\gamma =  g(x_0)\Psi(x_0)  +\int_{x_0}^\infty g' \Psi dx. 
  \end{equation}
In particular, $g' \Psi \in  L^1(x_0,\infty)$. Equation \ref{g'gauss} is also true with measures $\mu_n$ in the place of the Gaussian law $\gamma$:
\begin{equation}\label{g'binom}
\int_{x_0}^\infty g\, d\mu_n = g(x_0)\Psi_n(x_0)   +\int_{x_0}^\infty g' \Psi_n dx,\ \ \   n\in\NN.
\end{equation}
\end{prop}

\begin{proof} 
In order to prove (\ref{g'gauss}), we define $Y^{x_0}$ as a bounded and positive random variable with law
$\gamma|_{ [x_0,\infty)}/\gamma([x_0,\infty) )$.
By Fubini's theorem we write
 \begin{eqnarray*}
\displaystyle{ \frac{1}{\gamma([x_0,\infty))}\int_{x_0}^\infty g\, d\gamma} &= \displaystyle{\EE g(Y^{x_0})= \int_0^\infty P(g(Y^{x_0})>x)\,dx }\\
&= \displaystyle{\left(\int_0^{g(x_0)}+\int_{g(x_0)}^\infty \right) P(g(Y^{x_0})>x)dx }\\
&= \displaystyle{g(x_0) +\int_{g(x_0)}^\infty  P(g(Y^{x_0})>x)dx. }
   \end{eqnarray*}
 The function
 $g$ is a ${\mathcal C}^1$ bijection of $[x_0,\infty)$ onto $[g(x_0),G)$,
where $G=\lim_{x\ra\infty}g(x)$.
 In the last  integral we change the variables $u=g^{-1}(x)$ 
and we obtain
$$
\int_{g(x_0)}^\infty  P(g(Y^{x_0})>x)dx=\int_{g(x_0)}^G P(g(Y^{x_0})>x)dx=
\int_{x_0} ^\infty P(Y^{x_0}>u)g'(u)du
$$
and (\ref{g'gauss}) follows.  The proof for the symmetric binomial measures
$\mu_n$ is analogous.
\end{proof}

\begin{thm}\label{betterCLT}
If $g\in L^1(\gamma)$ is in ${\mathcal C}^1([0,\infty))$ and $g$
is {\bf strictly increasing} 
  on $[x_0,\infty)$ for an $x_0\ge 0$, 
 then the DeMoivre--Laplace CLT holds for $g$ and the subsequence $N=4n^2$:
\begin{eqnarray*}
\int_0^\infty g\,  d\mu_N &\ra& \int_0^\infty g\,   d\gamma,\ \ \ \ \ N=4n^2 \ra\infty.
\end{eqnarray*}

\end{thm}
\begin{proof}
By the Central Limit Theorem $\int_0^{x_0} g\,  d\mu_N \ra \int_0^{x_0} g\, d\gamma$
and $\Psi_n(x)\ra \Psi(x)$, $n\ra\infty$.
 In order to establish the convergence of integrals on $[x_0,\infty)$, we begin with the formula (\ref{g'binom}). The convergence of the term
$\int_{x_0}^\infty g' \Psi_N\, dx$ to $\int_{x_0}^\infty g' \Psi\, dx$
follows by the Dominated Convergence Theorem using Proposition \ref{tails} and the integrability of $g' \Psi$ with respect to Lebesgue measure on $[x_0,\infty)$. An application of (\ref{g'gauss}) ends the proof.
\end{proof}

 \begin{prop}\label{tails}
 Let $x_0>0$. There exists $C>0$ such that for all $x\ge x_0$ and $ N=4n^2 $
$$
\Psi_N(x)= P(S_N>x)\le C\,P(Y>x)=C \Psi(x).
$$
\end{prop}
\begin{rem}
Proposition \ref{tails} strengthens the  Hoeffding inequality (\ref{hoeff}).
 For our application, it is sufficient for us to prove it for a subsequence of $n$ (here $4n^2$),
but we conjecture that it is true for all $n$.
\end{rem}
\begin{proof}
 Let us denote $b(k,n,p)={n \choose k} p^k(1-p)^{n-k}$ and put $B(k,n,p)=\sum_{\nu=0}^k b(\nu,n,k).$  It is  a standard exercise (cf. \cite{f} Ex.VI.45(10.9)) to show that 
\[ 1-B(k,n,p)=n{{n-1} \choose k} \int_0^p t^k(1-t)^{n-k-1}\, dt. \]
We will show that if $p=\frac12$ and $k=\lfloor\frac{n}{2}+\frac{x\sqrt{n}}{2}\rfloor$ 
then there exists a constant $C$ such that for $x>x_0$ there holds
$1-B\left(k,n,\frac12\right) \ < \ C \Psi(x) $.
By the well-known estimate  $\Psi(x) \sim \frac1x e^{-x^2/2}$ (see e.g. \cite{f}VII, Lemma 2),
it is enough to show that for $x>x_0$
$$
1-B\left(k,n,\frac12\right) \ < \ \frac{C}x e^{-\frac{x^2}{2}}.$$
In order to simplify the left--hand side of the  last  inequality we  write
$$
\frac{1-B(k,n,\frac12)}{b(k,n,\frac12)}= 
\frac{n {{n-1} \choose k} \int_0^{1/2} t^k(1-t)^{n-k-1}\, dt}
{{{n} \choose k} \left(\frac{1}{2}\right)^n}=
(n-k)2^n\int_0^{\frac12} t^k (1-t)^{n-k+1} \, dt,$$
so that it is enough to show that 
\begin{equation}\label{raz}
b\left(k,n,\frac12\right) (n-k)2^n\int_0^{\frac12} t^k (1-t)^{n-k-1} \, dt \le
 \frac{C}{x}  e^{-\frac{x^2}{2}}.
 \end{equation}

In order to further simplify the computations, from now on we take a subsequence $N=4n^2$ instead of $n$, which gives $k=2n^2+ \lfloor xn \rfloor$.
For such $k$, Inequality \ref{raz} reads as
 \begin{equation}\label{dwa}
{4n^2 \choose 2n^2+\lfloor x n\rfloor} 
(2n^2-\lfloor x n\rfloor)
\int_0^{\frac12} 
t^{2n^2+\left\lfloor x n\right\rfloor} 
(1-t)^{2n^2- \lfloor x n \rfloor -1}
 \, dt
\le 
 \frac{C}x  e^{-\frac{x^2}{2}}.
\end{equation}
First we estimate the integral. For $0\le t \le \frac12$ there holds 
$\left(\rule{0cm}{3,5mm}t(1-t)\right)^{2n^2}<\left(\frac{1}{4}\right)^{2n^2}$.
In order to estimate the integral $\int_0^{\frac12} \left(\frac{t}{1-t}\right)^{\lfloor xn\rfloor}\, dt$
we use the Laplace method
for estimating integrals of type $\int_a^b \exp\left(\lambda S(x)\right)\, dx$, when $\lambda \to \infty$,
see e.g. \cite{O}.
We have to estimate
$$
\int_0^{\frac12} \left(\frac{t}{1-t}\right)^{\lfloor xn\rfloor}\, dt
=
\int_0^{\frac12} e^{{\lfloor xn\rfloor}\ln \frac{t}{1-t}}\, dt,$$
hence we take
$
\lambda=\lfloor xn\rfloor$ and $ S(t)= \ln \frac{t}{1-t}.$
If $S(x)$ is $\mathcal{C}^{\infty}$, $\max_{x\in [a,b]} S(x)$ is attained only at $b$ 
and $S'(b)\neq 0$ --- all these conditions are fulfilled in our case ---
then, by Laplace method, for  $\lambda \to \infty$, there holds
$
\int_a^b \exp\left(\lambda S(x)\right)\, dx \ \sim \ \frac{1}{\lambda S'(b)},$
which in our case gives for some  constant $C_1$ and $x>x_0$
$$
\int_{0}^{\frac12} 
\left(\frac{t}{1-t}\right)^{\lfloor x n \rfloor} \, dt
\le   
\frac{C_1}{x n}.$$
Finally,  since $(1-t)^{-1}\le 2$ on this interval, we get
$$
\int_0^{\frac12} 
t^{2n^2+\left\lfloor x n\right\rfloor} 
(1-t)^{2n^2- \lfloor x n \rfloor -1}
 \, dt\le \frac{2C_1}{xn}\left(\frac14\right)^{2n^2}.$$
Substituting  this estimate into (\ref{raz}), we see that it is enough to prove the following inequality:
there exists a constant $C_2$ such that for all $n\in \NN$ and all $x\in [x_0,\ 2n]$  there  holds
\begin{equation*}
\frac{2n^2-\lfloor x n\rfloor}{n}\,
{4n^2 \choose 2n^2+\lfloor x n\rfloor} 
\left(\frac14\right)^{2n^2}
\le 
C_2  e^{-\frac{x^2}{2}}.
\end{equation*}
If $x\in [x_0,\ 2n]$, then 
$m=\lfloor xn \rfloor  \in  \left[\rule{0cm}{3,8mm} \lfloor x_0 n \rfloor,\ 2n^2\right]  $, hence it is enough to show 
that for all $m=1,2,...,2n^2$ there holds
$$
\frac{2n^2- m}{n}\,
{4n^2 \choose 2n^2+ m } 
\left(\frac14\right)^{2n^2}
\le 
C_2  e^{-\frac{x^2}{2}}.$$
But,  if $m=\lfloor xn\rfloor$, then 
$
xn-1 < m \le xn, \ \  \mbox{hence}  \ \ 
x-\frac1n < \frac{m}{n} \le x \ \ \ \mbox{and} \ \
x <\frac{m+1}{n},$
which implies 
$ e^{-\frac{(m+1)^2}{2n^2}}< e^{-\frac{x^2}{2}}.$
Taking this into account, we see that
it is enough to prove that for all $n$ and  $m=1,2,...,2n^2$
\begin{equation}\label{four}
\frac{2n^2- m}{n}\,
{4n^2 \choose 2n^2+ m } 
\left(\frac14\right)^{2n^2}
\le 
C_2  e^{-\frac{(m+1)^2}{2n^2}}. 
\end{equation}

We estimate from the above the left-hand side of $(\ref{four})$, using the Stirling formula
\[ N!=N^N\,e^{-N}\sqrt{2\pi N} \,\exp\{\theta_N\}, \]
where $\theta_N\in (0,1)$ and $N\in\NN$. We obtain
 \begin{eqnarray*}
\frac{2n^2- m}{n}\,
{4n^2 \choose 2n^2+ m } 
\left(\frac14\right)^{2n^2}&=&
\frac{(2n^2- m)}{n}\,
\frac{(4n^2)!}{(2n^2+m )!(2n^2-m )!}
\left(\frac14\right)^{2n^2}\\
&\le&
\frac{2e}{\sqrt{2\pi}}
\left(\frac{2n^2- m}{2n^2+m}\right)^{\frac12}
\left(\frac{4n^4}{4n^4-m^2}\right)^{2n^2} 
\left(\frac{2n^2-m}{2n^2+m}\right)^m.
 \end{eqnarray*}

 We see that (\ref{four}) will follow from an estimate
$$
\left(\frac{4n^4}{4n^4-m^2}\right)^{2n^2}
\left(\frac{2n^2-m}{2n^2+m}\right)^{m+\frac{1}{2}}
\le
C_3  e^{-\frac{(m+1)^2}{2n^2}}$$
for some $C_3$ and $m=1,2,...,2n^2$, $n\in \NN$.
 We write

\[ \begin{aligned}
\left(\frac{4n^4}{4n^4-m^2}\right)^{2n^2} 
\left(\frac{2n^2-m}{2n^2+m}\right)^{m+\frac{1}{2}}
&=
\left(1-\frac{m^2}{4n^4}\right)^{-2n^2} 
\left(\frac{1-\frac{m}{2n^2}}{1+\frac{m}{2n^2}}
\right)^{m+\frac{1}{2}}\\
&=
\exp
\left(
-2n^2\ln\left(1- \frac{m^2}{4n^4}\right) +(m+1/2) \ln 
\frac{1-\frac{m}{2n^2}}{1+\frac{m}{2n^2}}
\right),
\end{aligned} \]

so that we have to prove that for all $n$ and $m=1,2,...,2n^2$ and some  constant $C_4=\ln C_3$
\begin{equation}\label{five}
-2n^2\ln \left(1- \frac{m^2}{4n^4}\right)
 +(m+1/2) 
\ln \frac{1-\frac{m}{2n^2}}{1+\frac{m}{2n^2}}
\le C_4-\frac{(m+1)^2}{2n^2}.
\end{equation}

Observe that if $m=2n^2$, then the left-hand side of (\ref{four}) is zero and then (\ref{four}) is obviously true.
For $m=1,2,...,2n^2-1$ the quantity $t=\frac{m}{2n^2}$ is positive and strictly less then one, so that we can
use the Taylor series expansions for $|t|<1$ and the functions $\ln(1-t^2)$, $\ln(1+t)$, and $\ln(1-t)$.
After some tedious but elementary computations one finds that the left-hand side $\ell$ of (\ref{five})
 has the form
$
\ell=- \frac{m^2+m}{2n^2}+R(n,m)
$
where $R(n,m)=\sum a_j(n) m^j$ is negative, because 
all the coefficients $a_j(n)$ are negative.
Now, the inequality 
$
 \ell \le C_4 -\frac{(m+1)^2}{2n^2}
$
 obviously follows 
because   
$
 -\frac{m^2+m}{2n^2}+\frac{(m+1)^2}{2}=
\frac{m+1}{2n^2}\le 1.
$
\end{proof}

\subsection{Proofs of Gaussian Log--Sobolev Inequalities}
We are now ready to prove the Theorems  \ref{LSItheoremGaussianL2}
and \ref{LSItheoremGaussianL1}. We present, with details, the proof
in the more natural $L^1$ case.

\begin{proof}[Proof of Theorem \ref{LSItheoremGaussianL1}]

By Theorem \ref{LSItheorem}  
we know that for all $n$ and $f\in{\mathcal C}^1$ 
\begin{equation}\label{LSIbin} 
\int  f\log f\, d\mu_n -\|f\|_{1,\mu_n} \log\|f\|_{1,\mu_n} \le \frac{1}{2}\int  Ef\, d\mu_n,
\end{equation}
where $\mu_n$ is the convolved Bernoulli measure considered in the previous section.  We want to show that the Central Limit Theorem with $n=(2k)^2$ applies to all the three terms of  the formula (\ref{LSIbin}).

\medskip

It is sufficient to show that the integrals $\int_0^\infty h\,d\mu_n$ restricted to $[0,\infty)$ converge to $\int_0^\infty h\,d\gamma$ for $h=f\log f, f$ and $Ef$. Indeed, using the notation $\tilde f(x)=f(-x)$, if $f$ and $\log f$ are convex, so are $\tilde f$ and $\log\tilde f$, so $f$ being log-subharmonic is equivalent to $\tilde f$ being log-subharmonic. The property $E\tilde f(x)=-xf'(-x)$ shows that on the right--hand side of (\ref{LSIbin}) and (\ref{LSIGaussL1}) we have $\int_{-\infty}^0 Ef(x)d\mu(x)=\int_0^\infty E\tilde f(x) d\mu(x)$.\\

{\it First term.} The function $\log f$ is ${\mathcal C}^1$  and convex, so it is monotone in a segment $[x_0,\infty)$.
\begin{itemize}
\item If $\lim_{x\ra\infty} \log f(x)=c$ is finite, then $\log f$ is bounded on
$[x_0,\infty)$ and therefore $f$ is bounded. Thus $f\log f$ is bounded on $[x_0,\infty)$
and on $[0,\infty)$. The convergence $\int f\log f\, d\mu_n\ra  \int f\log f\, d\gamma$ then follows from the CLT.
\item If $\lim_{x\ra\infty} \log f(x)=-\infty$, then $\lim_{x\ra\infty} f=0$
and $\lim_{x\ra\infty} f\log f=0$. As in the preceding case, 
the convergence  $\int f\log f\, d\mu_n\ra  \int f\log f \,d\gamma$   follows from the CLT.
\item In the case $\lim_{x\ra\infty} \log f(x)=+\infty$, the function $\log f$
is  increasing on  $[x_0,\infty)$, thus $f$ is also increasing on  $[x_0,\infty)$.
We can suppose that $\log f>0$ on $[x_0,\infty)$(otherwise we choose
$x_0$ bigger). Consequently $ f \log f$ is increasing on $[x_0,\infty)$. If 
  $ f$ is not constant,  the functions $f$ and $ f \log f$ are strictly increasing.
We can then apply Proposition \ref{betterCLT}.
\end{itemize}

{\it Second term.} As a positive convex   function, $f$ is bounded on $\RR^+$
or strictly increasing on an interval $[x_0,\infty)$.
The convergence $\int f\,  d\mu_n\ra  \int f\,  d\gamma$ follows respectively from the CLT
or from Proposition \ref{betterCLT}.

{\it Third term}. The function $f'$ is increasing. Therefore, if $f$ achieves any positive values then
 $f'>0$ on a certain interval  $[x_0,\infty)$. As the function $x$ is strictly 
increasing, so is the function $xf'$ on  $[x_0,\infty)$ and we apply Proposition  \ref{betterCLT}.  If, on the other hand, $f'\le 0$ on $[0,\infty)$, then there exists a constant $C$ such that $|f'|\le C$
on $\RR^+$. Consequently $|Ef(x)|\le Cx$ on $\RR^+$ and the convergence 
$\int_0^\infty Ef\, d\mu_n\ra \int_0^\infty Ef \,d\gamma$ follows from 
the  Theorem \ref{CLTepsilon}.
\end{proof}

\begin{proof}[Proof of Theorem \ref{LSItheoremGaussianL2}] 
The proof is the same as the proof of Theorem \ref{LSItheoremGaussianL1},
with $f^2$ instead of $f$. In particular, for the convergence of the third integral
$\int f Ef\,d\mu_n$, we have $f Ef=\frac{1}{2} E(f^2)$ and the reasoning
from the proof of Theorem \ref{LSItheoremGaussianL1} applies.
\end{proof}

\begin{rem} The preceding techniques clearly only apply in the one-dimensional setting.  With the techniques in this paper, we cannot address the question of whether the stronger (constant $1/2$) Logarithmic Sobolev inequality of Theorems \ref{LSItheoremGaussianL1} and \ref{LSItheoremGaussianL2} hold for Gaussian measures in higher dimensions.  In principle, they should follow from the strong hypercontractivity inequalities of Theorem \ref{hcs} via an approach like that in the proof of Theorem \ref{LSItheorem}.  As we have mentioned, there are challenging regularization issues (due to the nature of logarithmically subharmonic functions) which complicate these techniques.  Along the same lines, any measure for which the Logarithmic Sobolev Inequality holds for \lsh\ functions should also satisfy strong hypercontractive estimates (this was proved in the restricted context of holomorphic functions in \cite{g}).  These issues will be dealt with in a future publication.
\end{rem}




\end{document}